\documentclass[11pt]{article}

\usepackage[T1]{fontenc}
\usepackage{lmodern}
\usepackage[margin=1.10in]{geometry}
\usepackage{amsmath,amssymb,amsthm,mathtools}
\usepackage{booktabs}
\usepackage{microtype}
\usepackage{xurl}
\usepackage[hidelinks]{hyperref}
\hypersetup{
  pdftitle={An Explicit Counterexample to the Rank-Two Poisson Conjecture},
  pdfauthor={Christopher D. Long},
  pdfsubject={An explicit noninvertible endomorphism of the canonical Poisson algebra on two canonical pairs},
  pdfkeywords={Poisson Conjecture, canonical Poisson algebra, symplectic-form-preserving polynomial map, Jacobian Conjecture, Dixmier Conjecture, Weyl algebra}
}
\allowdisplaybreaks

\newcommand{\C}{\mathbb C}
\newcommand{\Q}{\mathbb Q}
\newcommand{\N}{\mathbb N}
\newcommand{\Aff}{\mathbb A}
\newcommand{\Pcal}{\mathcal P}
\newcommand{\Bcal}{\mathcal B}
\newcommand{\Weyl}{\mathcal A}
\newcommand{\PC}{\mathrm{PC}}
\newcommand{\JC}{\mathrm{JC}}
\newcommand{\DC}{\mathrm{DC}}
\newcommand{\gr}{\operatorname{gr}}
\newcommand{\id}{\operatorname{id}}

\newtheorem{theorem}{Theorem}[section]
\newtheorem{proposition}[theorem]{Proposition}
\newtheorem{lemma}[theorem]{Lemma}
\newtheorem{corollary}[theorem]{Corollary}
\theoremstyle{remark}
\newtheorem{remark}[theorem]{Remark}

\title{An Explicit Counterexample to the Rank-Two Poisson Conjecture}
\author{Christopher D. Long\\\texttt{galizur@gmail.com}}
\date{July 21, 2026}

\begin{document}
\maketitle

\begin{abstract}
Let
\[
  \Pcal_2=\C[x,q,p,z]
\]
carry the canonical Poisson bracket determined by
\(
  \{p,x\}=\{z,q\}=1
\)
and by the vanishing of the other brackets between distinct generators.
Here and throughout, ``rank two'' means two canonical pairs in the standard
indexing of the canonical Poisson algebras; thus there are four polynomial
generators and the Poisson tensor has geometric rank four.  We give explicit
polynomials
\[
  R,T,D,S\in\Q[x,q,p,z]
\]
satisfying
\[
  \{D,R\}=1,\qquad \{S,T\}=1,
  \qquad
  \{R,S\}=\{R,T\}=\{D,S\}=\{D,T\}=0,
\]
while
\[
  R=x(2-3xq).
\]
Consequently, the assignment
\(
  (x,q,p,z)\mapsto(R,T,D,S)
\)
defines a Poisson endomorphism of \(\Pcal_2\) that is not an automorphism.
This disproves the Poisson Conjecture for two canonical pairs, and hence for
every number of canonical pairs at least two.  The associated polynomial map
of \(\Aff^4\) preserves the canonical symplectic form, has Jacobian
determinant one, and has an explicit fiber consisting of exactly three points.
The proof uses a polynomial source coordinate system in which the symplectic
identity reduces to three displayed coefficient identities.  A separate
appendix uses the same four output polynomials and their
Hamiltonian duals to construct an explicit nonautomorphic endomorphism of the
fourth Weyl algebra.
\end{abstract}



\section{Introduction and rank convention}\label{sec:introduction}

For a positive integer \(n\), write
\begin{equation}\label{eq:canonical-algebra-n}
  \Pcal_n
  =\C[x_1,\dots,x_n,p_1,\dots,p_n]
\end{equation}
with canonical Poisson bracket
\begin{equation}\label{eq:canonical-bracket-n}
  \{f,g\}
  =\sum_{i=1}^n
   \left(\frac{\partial f}{\partial p_i}
     \frac{\partial g}{\partial x_i}
     -
     \frac{\partial f}{\partial x_i}
     \frac{\partial g}{\partial p_i}
   \right).
\end{equation}
Thus \(\{p_i,x_j\}=\delta_{ij}\).  In this paper, we use \emph{rank \(n\)}
to mean the index \(n\) in \eqref{eq:canonical-algebra-n}, equivalently the
number of canonical pairs.  This differs from the geometric rank of the Poisson
tensor, which is \(2n\).  In particular, the title's ``rank two'' means two
canonical pairs, four polynomial variables, and a Poisson tensor of rank four.
This is also the indexing convention of Adjamagbo and van den Essen, whose
canonical Poisson algebra \(\Pcal_n\) has \(2n\) polynomial generators
\cite[Section~1]{AdjamagboVanDenEssen}.

The assertion \(\PC(n)\) is that every Poisson endomorphism of \(\Pcal_n\) is
an automorphism.  Recall that \(\JC(3)\), the three-variable Jacobian
Conjecture, asserts that a polynomial map \(\Aff^3_\C\to\Aff^3_\C\) with
nonzero constant Jacobian determinant is a polynomial automorphism
\cite{Keller}.  The present construction begins with an explicit
counterexample to this three-variable assertion.  The rank-two Poisson
endomorphism is obtained by adjoining a polynomial source coordinate and a
Hamiltonian correction, and every Poisson identity is then verified directly.

The construction below grew from the three-variable polynomial map
\begin{align}
  \Pi(X,Y,W)
    &=(1+XY)^3W+Y^2(1+XY)(4+3XY),\notag\\
  \Sigma(X,Y,W)
    &=Y+3X(1+XY)^2W+3XY^2(4+3XY),\notag\\
  \mathcal R(X,Y,W)
    &=2X-3X^2Y-X^3W.\label{eq:announced-map}
\end{align}
This map was announced by Alp\"oge, whose post thanks Akhil Mathew for asking
about the problem and Fable for working on it \cite{Alpoge}.  Buzzard's
contemporaneous account reports that Mathew suggested the problem to Alp\"oge
and that Fable found the counterexample \cite{Buzzard}; the Ulam research note
similarly credits Mathew for asking the question and Fable for the work leading
to the example \cite[Introduction]{UlamJacobian}.  The Ulam note independently
verifies the determinant and the three-point collision and studies the fibers
and global geometry \cite[Theorem~3.1 and Section~4]{UlamJacobian}.  Its
Jacobian determinant is \(-2\), and the three points
\begin{equation}\label{eq:announced-core-points}
  \left(0,0,-\frac14\right),\qquad
  \left(1,-\frac32,\frac{13}{2}\right),\qquad
  \left(-1,\frac32,\frac{13}{2}\right)
\end{equation}
have common image \((-1/4,0,0)\).  The standard cotangent lift
\[
  (X,\Xi)\longmapsto
  \bigl(F(X),J_F(X)^{-\mathsf T}\Xi\bigr),
  \qquad F=(\Pi,\Sigma,\mathcal R),
\]
is polynomial because \(\det J_F=-2\), preserves the canonical one-form
\(\Xi^{\mathsf T}dX\), and is noninjective on the zero section.  Thus the
announced map already gives a Poisson counterexample on three canonical pairs.
The purpose of the present construction is to reach two canonical pairs.

The three-dimensional core used here is the determinant-one normalization
\begin{equation}\label{eq:normalized-core-intro}
  G=(\mathcal R,-\Pi/2,\Sigma).
\end{equation}
The four-dimensional Poisson endomorphism is obtained from this core by a
polynomial change of source coordinates and a Hamiltonian correction.  Every
identity needed below is proved directly, so the proof does not depend on an
external verification of \eqref{eq:announced-map}.

The main result is the following.

\begin{theorem}[Main theorem]\label{thm:main}
There are explicit polynomials \(R,T,D,S\in\Q[x,q,p,z]\) such that
\begin{equation}\label{eq:Phi-intro}
  \Phi:\Pcal_2\longrightarrow\Pcal_2,
  \qquad
  \Phi(x)=R,\quad \Phi(q)=T,\quad \Phi(p)=D,\quad \Phi(z)=S,
\end{equation}
is a Poisson endomorphism but not an automorphism.  Its associated point map
\[
  \Phi_{\mathrm{pt}}=(R,T,D,S):\Aff^4_\C\longrightarrow\Aff^4_\C
\]
satisfies
\[
  \det J\Phi_{\mathrm{pt}}=1,
\]
and the fiber over \((0,1/8,0,0)\) consists exactly of
\begin{equation}\label{eq:three-points-intro}
  \left(0,0,\frac1{24},-\frac18\right),\qquad
  \left(1,\frac23,\frac{247}{96},-\frac{89}{64}\right),\qquad
  \left(-1,-\frac23,\frac{247}{96},-\frac{89}{64}\right).
\end{equation}
Consequently, \(\PC(2)\) is false.
\end{theorem}

The formulas have rational coefficients and therefore define the same
counterexample over every field of characteristic zero.  We work over \(\C\)
to match the standard formulations of the Poisson and Jacobian conjectures.

\section{Poisson endomorphisms and symplectic maps}\label{sec:preliminaries}

From now on, order the generators as \((x,q,p,z)\) and write
\begin{equation}\label{eq:poisson-bracket}
  \{f,g\}=f_p g_x-f_x g_p+f_z g_q-f_q g_z.
\end{equation}
Thus \(\{p,x\}=\{z,q\}=1\), and the corresponding symplectic form is
\begin{equation}\label{eq:omega}
  \omega=dx\wedge dp+dq\wedge dz.
\end{equation}
We record the elementary criterion used below.  It is the four-variable case
of \cite[Theorem~2.1 and Corollary~2.1]{AdjamagboVanDenEssen}.

\begin{lemma}[Symplectic criterion]\label{lem:symplectic-criterion}
Let
\[
  F=(F_1,F_2,F_3,F_4)\in\C[x,q,p,z]^4.
\]
The substitution
\(
  (x,q,p,z)\mapsto(F_1,F_2,F_3,F_4)
\)
defines a Poisson endomorphism for \eqref{eq:poisson-bracket} if and only if
\begin{equation}\label{eq:symplectic-criterion}
  dF_1\wedge dF_3+dF_2\wedge dF_4
  =dx\wedge dp+dq\wedge dz.
\end{equation}
Equivalently,
\[
  \{F_3,F_1\}=1,\qquad \{F_4,F_2\}=1,
\]
and the four mixed brackets vanish.  Whenever these conditions hold,
\begin{equation}\label{eq:symplectic-det-one}
  \det JF=1.
\end{equation}
\end{lemma}

\begin{proof}
Let \(X=(x,q,p,z)^{\mathsf T}\) and
\[
  \Omega_0=
  \begin{pmatrix}
    0&0&1&0\\
    0&0&0&1\\
    -1&0&0&0\\
    0&-1&0&0
  \end{pmatrix}.
\]
Then \(\Omega_0\) is the matrix of \(\omega\), while the matrix of the generator
brackets is \(\Omega_0^{-1}=-\Omega_0\).  Equation
\eqref{eq:symplectic-criterion} is
\[
  (JF)^{\mathsf T}\Omega_0(JF)=\Omega_0.
\]
It gives the left inverse \(\Omega_0^{-1}(JF)^{\mathsf T}\Omega_0\) for \(JF\).
Since the coefficient ring is commutative and the matrix is square, \(JF\) is
invertible.  Taking inverses gives
\[
  (JF)\Omega_0^{-1}(JF)^{\mathsf T}=\Omega_0^{-1},
\]
whose entries are exactly the generator-bracket identities.  Since the
Poisson bracket is a biderivation, preservation on generators is equivalent
to preservation on the whole polynomial algebra.

Finally, in the coordinate order \((x,q,p,z)\),
\[
  \omega^2=-2\,dx\wedge dq\wedge dp\wedge dz,
\]
whereas the square of the left side of
\eqref{eq:symplectic-criterion} is
\[
  -2(\det JF)\,dx\wedge dq\wedge dp\wedge dz.
\]
Thus \(\det JF=1\).
\end{proof}

\section{The explicit polynomials}\label{sec:construction}

Set
\begin{equation}\label{eq:a-B-beta}
  a=1-3xq,
  \qquad
  B=3x^2p+2az,
  \qquad
  \beta=B-9q^2,
\end{equation}
and then define
\begin{equation}\label{eq:y-u}
  y=q-\frac{x\beta}{3},
  \qquad
  u=xy.
\end{equation}
The three core polynomials are
\begin{align}
  R&=2x-3x^2y-x^3\beta,\label{eq:R-def}\\
  S&=y+3x(1+u)^2\beta+3xy^2(4+3u),\label{eq:S-def}\\
  T&=-\frac12\left((1+u)^3\beta
      +y^2(1+u)(4+3u)\right).\label{eq:T-def}
\end{align}
The uncorrected fourth coordinate is
\begin{equation}\label{eq:D0-def}
  D_0=\frac{1+3xq}{2}\,p-3q^2z.
\end{equation}
Define the Hamiltonian correction
\begin{align}
  H={}&\frac{y^4}{20}(18u^2+78u+125)\notag\\
     &+\frac{3\beta y^2}{10}(u^3+5u^2+10u-5)\notag\\
     &-\frac{\beta^2}{6}(9u+2)-\frac{x^2\beta^3}{6},
     \label{eq:H-def}
\end{align}
and put
\begin{equation}\label{eq:D-def}
  D=D_0+H.
\end{equation}
All displayed expressions are polynomials over \(\Q\).  Their expanded sizes
are included for reference.

\begin{center}
\begin{tabular}{@{}lrrrrrrr@{}}
\toprule
polynomial & \(\beta\) & \(y\) & \(R\) & \(S\) & \(T\) & \(H\) & \(D\)\\
\midrule
nonzero terms & 4 & 5 & 2 & 22 & 47 & 137 & 139\\
total degree  & 3 & 4 & 3 & 11 & 15 & 23 & 23\\
\bottomrule
\end{tabular}
\end{center}

The first coordinate simplifies drastically.

\begin{lemma}\label{lem:R-factorization}
In the original variables \((x,q,p,z)\),
\begin{equation}\label{eq:R-factorization}
  R=x(2-3xq).
\end{equation}
\end{lemma}

\begin{proof}
Using \(y=q-x\beta/3\),
\[
  R
  =2x-3x^2\left(q-\frac{x\beta}{3}\right)-x^3\beta
  =2x-3x^2q.
\]
\end{proof}

We henceforth use \(\Phi\) for the algebra endomorphism in
\eqref{eq:Phi-intro} and \(\Phi_{\mathrm{pt}}=(R,T,D,S)\) for its associated
point map.

\section{A polynomial source coordinate system}\label{sec:source-coordinates}

The choice of \(D_0\) makes \((x,y,\beta,D_0)\) a polynomial coordinate
system.  Temporarily put
\begin{equation}\label{eq:QACM}
  Q=y+\frac{x\beta}{3},
  \qquad
  A=1-3xQ,
  \qquad
  C=\frac{1+3xQ}{2},
  \qquad
  M=\beta+9Q^2.
\end{equation}
After substituting \eqref{eq:a-B-beta}--\eqref{eq:y-u}, one has \(Q=q\),
\(A=a\), and
\begin{equation}\label{eq:source-linear-system}
  M=3x^2p+2Az,
  \qquad
  D_0=Cp-3Q^2z.
\end{equation}

\begin{proposition}[Polynomial source coordinates]\label{prop:source-automorphism}
The map
\begin{equation}\label{eq:Psi0-def}
  \Psi_0:\Aff^4\longrightarrow\Aff^4,
  \qquad
  (x,q,p,z)\longmapsto(x,y,\beta,D_0),
\end{equation}
is a polynomial automorphism.  Its inverse is
\begin{equation}\label{eq:Psi0-inverse}
\begin{split}
  q&=Q=y+\frac{x\beta}{3},\\
  p&=3Q^2(\beta+9Q^2)+2(1-3xQ)D_0,\\
  z&=\frac{1+3xQ}{2}(\beta+9Q^2)-3x^2D_0.
\end{split}
\end{equation}
Moreover,
\begin{equation}\label{eq:Psi0-det}
  \det J\Psi_0=-1.
\end{equation}
The map
\begin{equation}\label{eq:Psi-def}
  \Psi:(x,q,p,z)\longmapsto(x,y,\beta,D)
\end{equation}
is also a polynomial automorphism, with inverse obtained from
\eqref{eq:Psi0-inverse} by replacing \(D_0\) with \(D-H(x,y,\beta)\).
\end{proposition}

\begin{proof}
The equations in \eqref{eq:source-linear-system} form the linear system
\[
  \begin{pmatrix}
    3x^2&2A\\
    C&-3Q^2
  \end{pmatrix}
  \begin{pmatrix}p\\z\end{pmatrix}
  =
  \begin{pmatrix}M\\D_0\end{pmatrix}.
\]
Its determinant is
\[
  -9x^2Q^2-2AC
  =-9x^2Q^2-(1-3xQ)(1+3xQ)
  =-1.
\]
Consequently,
\[
  \begin{pmatrix}p\\z\end{pmatrix}
  =
  \begin{pmatrix}
    3Q^2&2A\\
    C&-3x^2
  \end{pmatrix}
  \begin{pmatrix}M\\D_0\end{pmatrix},
\]
which is precisely \eqref{eq:Psi0-inverse}.

For the determinant, first pass from \((x,q,p,z)\) to
\((x,q,\beta,D_0)\).  At fixed \((x,q)\), the relevant block is
\[
  \begin{pmatrix}
    3x^2&2(1-3xq)\\[2pt]
    (1+3xq)/2&-3q^2
  \end{pmatrix},
\]
whose determinant is \(-1\).  The change
\(q\mapsto y=q-x\beta/3\) is a triangular shear of determinant one.  Hence
\eqref{eq:Psi0-det} holds.  Finally,
\((x,y,\beta,D_0)\mapsto(x,y,\beta,D_0+H)\) is another triangular shear,
with inverse \(D_0=D-H\), proving the assertion for \(\Psi\).
\end{proof}

In particular, \(x,y,\beta\) are algebraically independent, so calculations
in \(\Q[X,Y,W]\) may be transported to the subalgebra
\(\Q[x,y,\beta]\subset\Pcal_2\) without ambiguity.

\section{The symplectic calculation}\label{sec:symplectic-calculation}

We compute \(\omega\) in the coordinates \((x,y,\beta,D_0)\).  Continue to
use \(Q,A,C,M\) from \eqref{eq:QACM}.  From
\eqref{eq:Psi0-inverse},
\begin{equation}\label{eq:pz-inverse-short}
  p=3Q^2M+2AD_0,
  \qquad
  z=CM-3x^2D_0.
\end{equation}
The required differentials are
\begin{align}
  dQ&=dy+\frac{\beta}{3}\,dx+\frac{x}{3}\,d\beta,
  \label{eq:dQ}\\
  dp&=-6D_0Q\,dx
     +\bigl(6Q(\beta+18Q^2)-6D_0x\bigr)dQ\notag\\
     &\hspace{2.6cm}+3Q^2d\beta+2A\,dD_0,
     \label{eq:dp}\\
  dz&=\left(\frac32QM-6xD_0\right)dx
     +\left(\frac32xM+18QC\right)dQ\notag\\
     &\hspace{2.6cm}+C\,d\beta-3x^2dD_0.
     \label{eq:dz}
\end{align}
For example,
\[
  d(3Q^2M)=3Q^2d\beta+6Q(\beta+18Q^2)dQ
\]
and
\[
  d(2AD_0)=2A\,dD_0-6D_0Q\,dx-6D_0x\,dQ,
\]
which gives \eqref{eq:dp}; the formula for \(dz\) follows similarly from
\(dC=\frac32(Q\,dx+x\,dQ)\) and \(dM=d\beta+18Q\,dQ\).

Substituting \eqref{eq:dp} and \eqref{eq:dz} into
\(\omega=dx\wedge dp+dQ\wedge dz\) gives
\begin{align}
  \omega={}&\frac92Q(\beta+21Q^2)\,dx\wedge dQ
       +3Q^2\,dx\wedge d\beta
       +C\,dQ\wedge d\beta\notag\\
      &+2A\,dx\wedge dD_0-3x^2\,dQ\wedge dD_0.
      \label{eq:omega-Q}
\end{align}
Indeed, the coefficient of \(dx\wedge dQ\) is
\[
  6Q(\beta+18Q^2)-6D_0x
  -\left(\frac32QM-6xD_0\right)
  =\frac92Q(\beta+21Q^2).
\]
By Lemma~\ref{lem:R-factorization},
\begin{equation}\label{eq:R-Q}
  R=x(2-3xQ),
  \qquad
  dR=2A\,dx-3x^2dQ.
\end{equation}
Thus the last two terms of \eqref{eq:omega-Q} equal
\(dR\wedge dD_0\).  Replacing \(dQ\) using \eqref{eq:dQ}, we obtain
\begin{equation}\label{eq:omega-split}
  \omega=dR\wedge dD_0+\Theta,
\end{equation}
where
\begin{equation}\label{eq:Theta-def}
  \Theta=A_1\,dx\wedge dy+A_2\,dx\wedge d\beta
         +A_3\,dy\wedge d\beta
\end{equation}
with
\begin{align}
  A_1&=\frac92Q(\beta+21Q^2),\label{eq:A1}\\
  A_2&=3Q^2+\frac{\beta(1+3xQ)}6
            +\frac32xQ(\beta+21Q^2),\label{eq:A2}\\
  A_3&=\frac{1+3xQ}{2}.\label{eq:A3}
\end{align}

The correction \(H\) was chosen so that \(\Theta\) is exactly the remaining
part of the target symplectic form.

\begin{proposition}[Coefficient identity]\label{prop:coefficient-identity}
For the polynomials \(R,S,T,H\) in
\eqref{eq:R-def}--\eqref{eq:H-def}, regarded as elements of
\(\Q[x,y,\beta]\),
\begin{equation}\label{eq:two-form-identity}
  dR\wedge dH+dT\wedge dS=\Theta.
\end{equation}
\end{proposition}

\begin{proof}
For \(f,g\in\Q[x,y,\beta]\), the coefficients of \(df\wedge dg\) are
\begin{align*}
  [dx\wedge dy](df\wedge dg)&=f_xg_y-f_yg_x,\\
  [dx\wedge d\beta](df\wedge dg)&=f_xg_\beta-f_\beta g_x,\\
  [dy\wedge d\beta](df\wedge dg)&=f_yg_\beta-f_\beta g_y.
\end{align*}
Direct differentiation and collection give
\begin{align}
 [dx\wedge dy](dR\wedge dH+dT\wedge dS)
   &={\frac12}(x\beta+3y)
      \bigl(7x^2\beta^2+42xy\beta+3\beta+63y^2\bigr),
      \label{eq:coef-xy}\\
 [dx\wedge d\beta](dR\wedge dH+dT\wedge dS)
   &=\frac16\bigl(7x^4\beta^3+63x^3y\beta^2+6x^2\beta^2\notag\\
   &\hspace{1.25cm}+189x^2y^2\beta+24xy\beta+\beta
      +189xy^3+18y^2\bigr),
      \label{eq:coef-xbeta}\\
 [dy\wedge d\beta](dR\wedge dH+dT\wedge dS)
   &=\frac12(x^2\beta+3xy+1).
      \label{eq:coef-ybeta}
\end{align}
Appendix~\ref{app:derivative-certificate} records every derivative used in
this collection.

Since \(3Q=x\beta+3y\), the second factor in
\eqref{eq:coef-xy} is
\[
  7(x\beta+3y)^2+3\beta
  =3(\beta+21Q^2),
\]
so \eqref{eq:coef-xy} is \(A_1\).  Expanding
\[
  6A_2
  =18Q^2+\beta(1+3xQ)+9xQ(\beta+21Q^2)
\]
after substituting \(3Q=x\beta+3y\) gives the numerator in
\eqref{eq:coef-xbeta}.  Finally,
\[
  \frac12(x^2\beta+3xy+1)
  =\frac12(1+3xQ)=A_3.
\]
The three coefficients agree with \eqref{eq:Theta-def}, proving
\eqref{eq:two-form-identity}.
\end{proof}

\begin{theorem}[Exact Poisson identities]\label{thm:poisson-identities}
The polynomials in Section~\ref{sec:construction} satisfy
\begin{equation}\label{eq:six-brackets}
\begin{gathered}
  \{D,R\}=1,\qquad \{S,T\}=1,\\
  \{R,S\}=\{R,T\}=\{D,S\}=\{D,T\}=0.
\end{gathered}
\end{equation}
Consequently, \(\Phi\) is a Poisson endomorphism and
\begin{equation}\label{eq:Phi-det-one}
  \det J\Phi_{\mathrm{pt}}=1.
\end{equation}
\end{theorem}

\begin{proof}
By \eqref{eq:omega-split}, Proposition~\ref{prop:coefficient-identity}, and
\(D=D_0+H\),
\[
  \omega
  =dR\wedge dD_0+dR\wedge dH+dT\wedge dS
  =dR\wedge dD+dT\wedge dS.
\]
Apply Lemma~\ref{lem:symplectic-criterion} with
\((F_1,F_2,F_3,F_4)=(R,T,D,S)\).
\end{proof}

There is also a short independent check of the first canonical bracket.  From
\eqref{eq:R-factorization} and \eqref{eq:D0-def},
\[
  R_x=2(1-3xq),\qquad R_q=-3x^2,
  \qquad
  (D_0)_p=\frac{1+3xq}{2},\qquad (D_0)_z=-3q^2.
\]
Hence
\begin{equation}\label{eq:D0R-direct}
  \{D_0,R\}
  =\frac{1+3xq}{2}\,2(1-3xq)+(-3q^2)(-3x^2)=1.
\end{equation}
Proposition~\ref{prop:induced-bracket} below shows that \(R\) is a Casimir on
\(\C[x,y,\beta]\); since \(H\) lies in that subalgebra, \(\{H,R\}=0\).
Thus \(\{D,R\}=1\).

\section{The three-dimensional core}\label{sec:three-dimensional-core}

This section records the structural relation with
\eqref{eq:announced-map}.  Regard \(R,T,S\) as polynomials in the independent
variables \((x,y,\beta)\), and set
\begin{equation}\label{eq:G-core}
  G=(R,T,S):\Aff^3_{x,y,\beta}\longrightarrow\Aff^3.
\end{equation}

\begin{proposition}\label{prop:core-jacobian}
One has
\begin{equation}\label{eq:core-jacobian}
  \det\frac{\partial(R,T,S)}{\partial(x,y,\beta)}=1.
\end{equation}
\end{proposition}

\begin{proof}
Put
\begin{equation}\label{eq:w-alpha}
  w=1+xy,
  \qquad
  \alpha=2-3xy-x^2\beta.
\end{equation}
Then
\begin{equation}\label{eq:source-jac-core}
  \det\frac{\partial(x,w,\alpha)}{\partial(x,y,\beta)}=-x^3.
\end{equation}
In the localization \(\Q[x,x^{-1},w,\alpha]\), define
\begin{equation}\label{eq:b-c}
  b=2+4w-3\alpha w^2,
  \qquad
  c=\frac12(\alpha w^3-w^2-w).
\end{equation}
Substitution into \eqref{eq:R-def}--\eqref{eq:T-def} gives
\begin{equation}\label{eq:RST-localized}
  R=x\alpha,
  \qquad
  T=x^{-2}c,
  \qquad
  S=x^{-1}b.
\end{equation}
Therefore
\begin{equation}\label{eq:target-jac-core}
  \det\frac{\partial(R,T,S)}{\partial(x,w,\alpha)}
  =x^{-3}\Delta,
\end{equation}
where
\[
  \Delta=
  \det
  \begin{pmatrix}
    \alpha&0&1\\
    -2c&c_w&c_\alpha\\
    -b&b_w&b_\alpha
  \end{pmatrix}.
\]
From \eqref{eq:b-c},
\[
  b_w=4-6\alpha w,
  \quad b_\alpha=-3w^2,
  \quad
  c_w=\frac12(3\alpha w^2-2w-1),
  \quad c_\alpha=\frac12w^3.
\]
Expansion along the first row gives
\begin{align*}
  \Delta
  &=\alpha(c_wb_\alpha-c_\alpha b_w)+bc_w-2cb_w\\
  &=-\frac{\alpha w^2}{2}(3\alpha w^2-2w-3)\\
  &\quad+\frac12(3\alpha^2w^4-2\alpha w^3-3\alpha w^2-2)\\
  &=-1.
\end{align*}
Combining \eqref{eq:source-jac-core} and
\eqref{eq:target-jac-core} proves \eqref{eq:core-jacobian} in the
localization.  Both sides are polynomials, so the identity holds in
\(\Q[x,y,\beta]\).
\end{proof}

The induced three-variable Poisson structure makes the role of \(R\)
transparent.

\begin{proposition}[Induced Jacobian Poisson bracket]\label{prop:induced-bracket}
On \(\C[x,y,\beta]\subset\Pcal_2\),
\begin{equation}\label{eq:induced-generator-brackets}
  \{x,y\}=x^3,
  \qquad
  \{x,\beta\}=-3x^2,
  \qquad
  \{y,\beta\}=-2+6xy+3x^2\beta.
\end{equation}
For all \(f,g\in\C[x,y,\beta]\),
\begin{equation}\label{eq:jacobian-bracket}
  \{f,g\}
  =-\det\frac{\partial(R,f,g)}{\partial(x,y,\beta)}.
\end{equation}
In particular, \(R\) is a Casimir of this three-variable Poisson algebra, and
\[
  \{S,T\}=1,
  \qquad
  \{R,S\}=\{R,T\}=0.
\]
\end{proposition}

\begin{proof}
Since \(\beta_p=3x^2\) and \(\beta_z=2(1-3xq)\),
\[
  \{x,\beta\}=-3x^2,
  \qquad
  \{q,\beta\}=-2+6xq.
\]
Using \(y=q-x\beta/3\) gives the remaining identities in
\eqref{eq:induced-generator-brackets}.  Moreover,
\[
  R_x=2-6xy-3x^2\beta,
  \qquad
  R_y=-3x^2,
  \qquad
  R_\beta=-x^3.
\]
The bivector encoded by \eqref{eq:induced-generator-brackets} is
\[
  -R_\beta\,\partial_x\wedge\partial_y
  +R_y\,\partial_x\wedge\partial_\beta
  -R_x\,\partial_y\wedge\partial_\beta,
\]
which is exactly \eqref{eq:jacobian-bracket}.  The claims involving \(R\)
follow immediately.  Finally, Proposition~\ref{prop:core-jacobian} gives
\[
  \{S,T\}
  =-\det\frac{\partial(R,S,T)}{\partial(x,y,\beta)}=1.
\]
\end{proof}

The point map factors as
\begin{equation}\label{eq:point-factorization}
  \Phi_{\mathrm{pt}}
  =\sigma\circ(G\times\id)\circ\Psi,
  \qquad
  \sigma(r,t,s,e)=(r,t,e,s).
\end{equation}
Indeed, \(\Psi\) produces \((x,y,\beta,D)\), then
\(G\times\id\) produces \((R,T,S,D)\), and \(\sigma\) interchanges the last
two coordinates.  The determinant factors are
\[
  \det J\Psi=-1,
  \qquad
  \det J(G\times\id)=1,
  \qquad
  \det J\sigma=-1,
\]
which gives a second proof of \eqref{eq:Phi-det-one}.  This factorization also
explains the inherited three-point fiber; the two-form calculation in
Section~\ref{sec:symplectic-calculation} is what establishes the full Poisson
property.

\section{Noninvertibility and the exact three-point fiber}\label{sec:noninvertibility}

\begin{proposition}\label{prop:not-automorphism}
The Poisson endomorphism \(\Phi\) is not an algebra automorphism.
\end{proposition}

\begin{proof}
By Lemma~\ref{lem:R-factorization},
\[
  \Phi(x)=R=x(2-3xq).
\]
The two factors are comaximal because
\begin{equation}\label{eq:bezout}
  1=\frac12(2-3xq)+\frac{3q}{2}x.
\end{equation}
The Chinese remainder theorem therefore gives
\begin{align}
  \Pcal_2/(R)
  &\cong \Pcal_2/(x)\times\Pcal_2/(2-3xq)\notag\\
  &\cong \C[q,p,z]\times\C[x,x^{-1},p,z].
  \label{eq:CRT}
\end{align}
This ring is not a domain, so \((R)\) is not prime.  By contrast,
\(\Pcal_2/(x)\cong\C[q,p,z]\) is a domain.  An automorphism would send the
prime ideal \((x)\) to the prime ideal \((\Phi(x))=(R)\), contradicting
\eqref{eq:CRT}.
\end{proof}

Equivalently, a polynomial-ring automorphism sends an irreducible element to
an irreducible element up to a unit, whereas \(R\) is a product of two
nonunits.  The point map also has a completely explicit failure of
injectivity.

\begin{proposition}[Exact fiber]\label{prop:exact-fiber}
The fiber of \(\Phi_{\mathrm{pt}}\) over
\begin{equation}\label{eq:fiber-target}
  (R,T,D,S)=\left(0,\frac18,0,0\right)
\end{equation}
consists exactly of the three points in \eqref{eq:three-points-intro}.
\end{proposition}

\begin{proof}
Because \(\Psi=(x,y,\beta,D)\) is a polynomial automorphism and
\(R,T,S\) depend only on \(x,y,\beta\), it is enough first to solve
\begin{equation}\label{eq:core-fiber-equations}
  R=0,
  \qquad
  S=0,
  \qquad
  T=\frac18
\end{equation}
in \((x,y,\beta)\), with \(D=0\).

If \(x=0\), then \eqref{eq:S-def} gives \(S=y\), so \(y=0\); then
\eqref{eq:T-def} gives \(T=-\beta/2\), whence \(\beta=-1/4\).

Suppose \(x\neq0\), and use \(w,\alpha\) from \eqref{eq:w-alpha}.  The
localized formulas \eqref{eq:RST-localized} read
\begin{equation}\label{eq:localized-fiber-formulas}
\begin{split}
  R&=x\alpha,\\
  S&=x^{-1}(2+4w-3\alpha w^2),\\
  T&=\frac1{2x^2}(\alpha w^3-w^2-w).
\end{split}
\end{equation}
The first two equations of \eqref{eq:core-fiber-equations} give
\(\alpha=0\) and \(w=-1/2\).  The third formula then becomes
\[
  T=\frac1{8x^2}.
\]
Since \(T=1/8\), one has \(x^2=1\).  Moreover,
\[
  xy=w-1=-\frac32,
  \qquad
  x^2\beta=2-3xy=\frac{13}{2}.
\]
Thus the complete fiber in \((x,y,\beta,D)\)-coordinates is
\begin{equation}\label{eq:core-three-point-fiber}
  \left(0,0,-\frac14,0\right),
  \qquad
  \left(1,-\frac32,\frac{13}{2},0\right),
  \qquad
  \left(-1,\frac32,\frac{13}{2},0\right).
\end{equation}

It remains to recover \((q,p,z)\).  Direct substitution in
\eqref{eq:H-def} gives
\begin{equation}\label{eq:H-fiber-values}
  H\left(0,0,-\frac14\right)=-\frac1{48},
  \qquad
  H\left(\pm1,\mp\frac32,\frac{13}{2}\right)
  =-\frac{1097}{192}.
\end{equation}
Since \(D=0\), the coordinate \(D_0=D-H\) equals \(1/48\) at the first
point and \(1097/192\) at the other two.  At
\((x,y,\beta)=(0,0,-1/4)\), the inverse formulas
\eqref{eq:Psi0-inverse} give
\[
  q=0,
  \qquad
  p=\frac1{24},
  \qquad
  z=-\frac18.
\]
At \((x,y,\beta)=(1,-3/2,13/2)\), one has
\[
  Q=\frac23,
  \qquad
  A=-1,
  \qquad
  C=\frac32,
  \qquad
  M=\frac{21}{2},
\]
and hence
\begin{align*}
  p&=3\left(\frac23\right)^2\frac{21}{2}
      -2\left(\frac{1097}{192}\right)
     =\frac{247}{96},\\
  z&=\frac32\frac{21}{2}
      -3\left(\frac{1097}{192}\right)
     =-\frac{89}{64}.
\end{align*}
The case \((-1,3/2,13/2)\) gives \(Q=-2/3\) and the same values of
\(p,z\).  These are exactly the points in \eqref{eq:three-points-intro}.
\end{proof}

\begin{remark}[Gr\"obner-basis certificate]\label{rem:fiber-groebner}
Let
\[
  I=\left(R,S,T-\frac18\right)\subset\Q[x,y,\beta],
\]
where \(R,S,T\) are regarded as polynomials in the independent variables
\((x,y,\beta)\).  With lexicographic order \(\beta>y>x\), exact reduction gives
the reduced Gr\"obner basis
\begin{equation}\label{eq:fiber-groebner-basis}
  \beta-\frac{27x^2-1}{4},
  \qquad
  y+\frac{3x}{2},
  \qquad
  x(x-1)(x+1).
\end{equation}
This independently certifies that the core fiber is exactly the three points
in \eqref{eq:core-three-point-fiber}.  Since the final polynomial in
\eqref{eq:fiber-groebner-basis} is squarefree in characteristic zero, the fiber
scheme is reduced.
\end{remark}

Theorem~\ref{thm:main} now follows from
Theorem~\ref{thm:poisson-identities}, Proposition~\ref{prop:not-automorphism},
and Proposition~\ref{prop:exact-fiber}.

\section{Extension to higher Poisson rank}\label{sec:higher-ranks}

\begin{corollary}\label{cor:all-poisson-ranks}
The Poisson Conjecture \(\PC(n)\) is false for every \(n\geq2\).
\end{corollary}

\begin{proof}
For \(n>2\), extend \(\Phi\) by the identity on the remaining canonical
pairs.  The extended map is Poisson, while the irreducible generator \(x\)
still has the reducible image \(x(2-3xq)\); hence the extension cannot be an
automorphism.
\end{proof}

Appendix~\ref{app:DC4} gives a separate explicit Hamiltonian construction in
the fourth Weyl algebra.  It is not used in the proof of the Poisson theorem:
the appendix directly verifies the Weyl relations and proves non-surjectivity.
Its purpose is the explicit Hamiltonian witness attached to the rank-two
Poisson map, rather than merely the existence of a counterexample at that Weyl
index.

\section{Exact computational audit}\label{sec:verification}

The proof above is algebraic and does not depend on computer algebra.  Two
exact programs accompany this source as independent transcription and
expansion checks.

The first, \texttt{verify\_rank2\_poisson\_sympy.py}, uses SymPy over
\(\Q\) \cite{SymPy}.  It verifies the factorization
\eqref{eq:R-factorization}, all six brackets in \eqref{eq:six-brackets}, the
source inverse \eqref{eq:Psi0-inverse}, all three coefficients in
\eqref{eq:two-form-identity}, the core determinant
\eqref{eq:core-jacobian}, the induced brackets
\eqref{eq:induced-generator-brackets}, the localized formulas
\eqref{eq:RST-localized}, the three-point fiber, the Gr\"obner basis
\eqref{eq:fiber-groebner-basis}, and the size data in
Section~\ref{sec:construction}.  It verifies the four-variable determinant
through the factorization \eqref{eq:point-factorization}, rather than by a
needlessly large direct determinant expansion.  For Appendix~\ref{app:DC4},
it assembles the full matrix \((\delta_i(f_j))\) from the verified brackets
and checks that all pairwise brackets of the four Hamiltonians are constant;
these are precisely the identities behind
\(\mathsf A J_f^{\mathsf T}=I_4\) and the commuting derivations.

The second, \texttt{verify\_rank2\_poisson\_sparse.py}, implements sparse
multivariate polynomials and exact rational arithmetic using only Python's
standard library.  It imports no computer algebra system and independently
checks the six brackets, source inverse, coefficient identity, core
determinant, collision, and size data.  Both programs return zero residuals
for every identity encoded in their respective audit suites.

\appendix

\section{Derivative certificate for the coefficient identity}
\label{app:derivative-certificate}

This appendix supplies the derivative list behind
Proposition~\ref{prop:coefficient-identity}.  Put
\[
  w=1+u.
\]
Define
\[
  P_1(u)=18u^2+78u+125,
  \qquad
  P_2(u)=u^3+5u^2+10u-5.
\]
Then
\[
  P_1'(u)=36u+78,
  \qquad
  P_2'(u)=3u^2+10u+10.
\]
Treating \(x,y,\beta\) as independent variables, differentiation gives
\begin{align}
  R_x&=2-6u-3x^2\beta,
  &R_y&=-3x^2,
  &R_\beta&=-x^3,
  \label{eq:R-derivatives-app}\\
  S_x&=3\beta(1+u)(1+3u)+6y^2(2+3u),
  \label{eq:Sx-app}\\
  S_y&=1+6x^2\beta(1+u)+24u+27u^2,
  &S_\beta&=3x(1+u)^2,
  \label{eq:Sy-Sbeta-app}\\
  T_x&=-\frac12\bigl(3\beta y(1+u)^2+y^3(7+6u)\bigr),
  \label{eq:Tx-app}\\
  T_y&=-\frac12\bigl(3\beta x(1+u)^2
       +2y(1+u)(4+3u)+xy^2(7+6u)\bigr),
  \label{eq:Ty-app}\\
  T_\beta&=-\frac12(1+u)^3,
  \label{eq:Tbeta-app}\\
  H_x&=\frac{y^5}{20}P_1'(u)
       +\frac{3\beta y^3}{10}P_2'(u)
       -\frac32\beta^2y-\frac13x\beta^3,
  \label{eq:Hx-app}\\
  H_y&=\frac{4y^3P_1(u)+xy^4P_1'(u)}{20}\notag\\
     &\quad+\frac{3\beta}{10}
       \bigl(2yP_2(u)+xy^2P_2'(u)\bigr)
       -\frac32x\beta^2,
  \label{eq:Hy-app}\\
  H_\beta&=\frac{3y^2}{10}P_2(u)
       -\frac{\beta(9u+2)}3-\frac12x^2\beta^2.
  \label{eq:Hbeta-app}
\end{align}
Substituting \eqref{eq:R-derivatives-app}--\eqref{eq:Hbeta-app} into
\begin{align*}
  C_{xy}&=R_xH_y-R_yH_x+T_xS_y-T_yS_x,\\
  C_{x\beta}&=R_xH_\beta-R_\beta H_x
                +T_xS_\beta-T_\beta S_x,\\
  C_{y\beta}&=R_yH_\beta-R_\beta H_y
                +T_yS_\beta-T_\beta S_y
\end{align*}
and collecting terms gives
\begin{align*}
  C_{xy}
  &=\frac12(x\beta+3y)
      \bigl(7x^2\beta^2+42xy\beta+3\beta+63y^2\bigr),\\
  C_{x\beta}
  &=\frac16\bigl(7x^4\beta^3+63x^3y\beta^2+6x^2\beta^2
      +189x^2y^2\beta\\
  &\hspace{25mm}{}+24xy\beta+\beta+189xy^3+18y^2\bigr),\\
  C_{y\beta}
  &=\frac12(x^2\beta+3xy+1).
\end{align*}
These are exactly \eqref{eq:coef-xy}--\eqref{eq:coef-ybeta}.  The reductions
to \(A_1,A_2,A_3\) are carried out in the proof of
Proposition~\ref{prop:coefficient-identity}.

\section{An explicit nonautomorphic endomorphism of the fourth Weyl algebra}
\label{app:DC4}

Let
\[
  \Bcal=\C[X_1,X_2,X_3,X_4]
\]
and identify \((X_1,X_2,X_3,X_4)\) with \((x,q,p,z)\).  In this appendix,
the subscript four counts the four commuting coordinate generators; this is
distinct from the canonical-pair convention used for the phrase ``rank two''
in the title.  The fourth Weyl algebra is
\[
  \Weyl_4
  =\C\langle X_1,\dots,X_4,\partial_1,\dots,\partial_4\rangle,
\]
with relations
\[
  [\partial_i,X_j]=\delta_{ij},
  \qquad
  [X_i,X_j]=[\partial_i,\partial_j]=0.
\]
It acts faithfully as the ring of polynomial differential operators on
\(\Bcal\).  The conjecture \(\DC(4)\) asserts that every endomorphism of
\(\Weyl_4\) is an automorphism; this is the fourth-rank instance of Dixmier's
problem \cite{Dixmier}.

The existence of a nonautomorphic endomorphism of \(\Weyl_4\) can also be
obtained from the announced three-variable Keller map through the standard
Jacobian-to-Weyl lift and extension by the identity; see
\cite[equations~(1)--(2) and Proposition~5(3)]{Bavula}.  The purpose here is
more explicit: we construct a Hamiltonian endomorphism canonically attached to
the rank-two Poisson map and prove its non-surjectivity directly.

For \(F\in\Bcal\), let
\begin{equation}\label{eq:Hamiltonian-derivation}
  H_F=\{F,\mathord\cdot\}
  =F_p\partial_x+F_z\partial_q-F_x\partial_p-F_q\partial_z.
\end{equation}
Set
\begin{equation}\label{eq:f-vector-DC4}
  (f_1,f_2,f_3,f_4)=(R,T,D,S)
\end{equation}
and
\begin{equation}\label{eq:delta-vector-DC4}
  \delta_1=H_D,
  \qquad
  \delta_2=H_S,
  \qquad
  \delta_3=-H_R,
  \qquad
  \delta_4=-H_T.
\end{equation}
Each \(f_i\) acts by multiplication, and each \(\delta_i\) is an explicit
first-order differential operator with polynomial coefficients.

\begin{theorem}\label{thm:explicit-DC4}
The assignment
\begin{equation}\label{eq:widehat-Phi}
  \widehat\Phi:\Weyl_4\longrightarrow\Weyl_4,
  \qquad
  X_i\longmapsto f_i,
  \qquad
  \partial_i\longmapsto\delta_i
  \quad(1\leq i\leq4),
\end{equation}
defines an algebra endomorphism that is not an automorphism.  Consequently,
\(\DC(4)\) is false.
\end{theorem}

\begin{proof}
The six identities in \eqref{eq:six-brackets} give
\begin{equation}\label{eq:delta-fi}
  \delta_i(f_j)=\delta_{ij}.
\end{equation}
For example,
\(
  \delta_1(R)=\{D,R\}=1
\)
and
\(
  \delta_3(D)=-\{R,D\}=1
\).
Therefore
\[
  [\delta_i,f_j]=\delta_{ij}
\]
as differential operators on \(\Bcal\).  The multiplication operators
\(f_i\) commute.  Hamiltonian derivations satisfy
\begin{equation}\label{eq:Hamiltonian-commutator}
  [H_F,H_G]=H_{\{F,G\}},
\end{equation}
and every Poisson bracket between the Hamiltonians in
\eqref{eq:delta-vector-DC4} is constant.  Hence
\([\delta_i,\delta_j]=0\).  The images in \eqref{eq:widehat-Phi} satisfy all
defining Weyl relations, so \(\widehat\Phi\) is an endomorphism.  This
Hamiltonian lift is the present specialization of the construction in the
proof of \cite[Theorem~7(2), equation~(5)]{Bavula}; all relations needed here
have just been verified directly.

It remains to prove non-surjectivity.  Filter \(\Weyl_4\) by differential
order and write
\begin{equation}\label{eq:delta-coefficient-matrix}
  \delta_i=\sum_{k=1}^4 a_{ik}\partial_k,
  \qquad
  \mathsf A=(a_{ik}).
\end{equation}
Let \(J_f\) be the Jacobian matrix of \((f_1,f_2,f_3,f_4)\) with respect to
\((X_1,X_2,X_3,X_4)\).  Equation \eqref{eq:delta-fi} is the matrix identity
\begin{equation}\label{eq:A-J-identity}
  \mathsf A J_f^{\mathsf T}=I_4.
\end{equation}
Since \(\det J_f=1\), the matrix \(\mathsf A\) is invertible over \(\Bcal\).
The same Jacobian identity also implies that the \(f_i\) are algebraically
independent.  Indeed, a nonzero relation of minimal total degree would, after
differentiation and multiplication by \(J_f^{-1}\), give lower-degree
relations from all its first derivatives; minimality would force those
derivatives to be zero, so the original relation would be constant.  For the
order filtration, let \(\sigma_m\) denote the degree-\(m\) principal-symbol map,
and put
\[
  \eta_i=\sigma_1(\delta_i),
  \qquad
  \xi=(\xi_1,\dots,\xi_4)^{\mathsf T},
  \qquad
  \eta=(\eta_1,\dots,\eta_4)^{\mathsf T}.
\]
In the associated graded algebra
\begin{equation}\label{eq:gr-Weyl}
  \gr\Weyl_4\cong\Bcal[\xi_1,\dots,\xi_4],
\end{equation}
one has \(\eta=\mathsf A\xi\).  Since \(\mathsf A\in\mathrm{GL}_4(\Bcal)\),
\[
  \Bcal[\eta_1,\dots,\eta_4]
  =\Bcal[\xi_1,\dots,\xi_4].
\]
In particular, the monomials \(\eta^\beta\) are linearly independent over
\(\Bcal\), and
\begin{equation}\label{eq:principal-symbol-monomial}
  \sigma_{|\beta|}\!\left(f^\alpha\delta^\beta\right)
  =f^\alpha\eta^\beta.
\end{equation}

Assume for contradiction that \(\widehat\Phi\) is an automorphism.  Applying
it to the standard Poincar\'e--Birkhoff--Witt basis shows that
\begin{equation}\label{eq:new-PBW}
  \{f^\alpha\delta^\beta:\alpha,\beta\in\N^4\}
\end{equation}
is a \(\C\)-basis of \(\Weyl_4\).  Let \(g\in\Bcal\) be a multiplication
operator and write its unique finite expansion as
\[
  g=\sum_{\alpha,\beta}c_{\alpha\beta}f^\alpha\delta^\beta.
\]
Suppose
\[
  m=\max\{|\beta|:c_{\alpha\beta}\neq0
       \text{ for some }\alpha\}
\]
is positive.  Since \(g\) has differential order zero, the order-\(m\)
principal symbol must vanish:
\begin{equation}\label{eq:top-symbol}
  \sum_{|\beta|=m}
    \left(\sum_\alpha c_{\alpha\beta}f^\alpha\right)\eta^\beta=0.
\end{equation}
By the linear independence established above, each coefficient
\(\sum_\alpha c_{\alpha\beta}f^\alpha\) is zero.  By the
algebraic independence just noted, every scalar \(c_{\alpha\beta}\) with
\(|\beta|=m\) is zero, a contradiction.
Therefore every multiplication operator belongs to
\(\C[f_1,f_2,f_3,f_4]\), and
\begin{equation}\label{eq:B-equals-Cf}
  \Bcal=\C[f_1,f_2,f_3,f_4].
\end{equation}
Since the \(f_i\) are algebraically independent, the polynomial endomorphism
\((X_1,X_2,X_3,X_4)\mapsto(f_1,f_2,f_3,f_4)\) is injective; equation
\eqref{eq:B-equals-Cf} makes it surjective.  It would therefore be an
automorphism, contrary to Proposition~\ref{prop:not-automorphism}.  Thus
\(\widehat\Phi\) is not an automorphism.
\end{proof}

\begin{remark}[Why the fourth Weyl algebra appears]
The main construction has four algebraically independent output coordinates
\(f_1,f_2,f_3,f_4\).  Pairing them with the four Hamiltonian derivations in
\eqref{eq:delta-vector-DC4} produces four multiplication--derivation pairs,
which is precisely the generator count for \(\Weyl_4\).  Thus the same
polynomials naturally connect the rank-two canonical Poisson algebra---where
``rank two'' counts two canonical pairs---with the fourth Weyl algebra, whose
index counts four polynomial coordinates and their four conjugate
derivations.
\end{remark}

\section*{AI provenance, use, and author responsibility}

The three-variable core used here was announced by Levent Alp\"oge.  As
documented in Section~\ref{sec:introduction}, the announcement and
contemporaneous accounts credit Akhil Mathew with raising the problem and
Fable with the work leading to the counterexample.  The four-variable
construction, including the Hamiltonian correction \(H\), was produced during
an interactive research session with {ChatGPT 5.6 Sol}.  {ChatGPT 5.6 Sol}
also assisted with the differential-form organization of the proof, exact
symbolic verification, literature checking, and drafting of this manuscript.
{Claude Fable 5} subsequently supplied independent algebraic audits and
editorial comments.  The human author bears
full responsibility for checking the mathematics, exposition, source
attribution, and every claim in this manuscript.

\end{document}